\pdfoutput=1
\RequirePackage{ifpdf}
\ifpdf 
\documentclass[pdftex]{sigma}
\else
\documentclass{sigma}
\fi

\usepackage[all]{xy}

\numberwithin{equation}{section}

\newtheorem{Theorem}{Theorem}[section]
\newtheorem*{t*}{Arnold--Thom conjecture}

\newtheorem{Lemma}[Theorem]{Lemma}
\newtheorem{Proposition}[Theorem]{Proposition}
 { \theoremstyle{definition}
\newtheorem{Definition}[Theorem]{Definition}

\newtheorem{Example}[Theorem]{Example}
\newtheorem{examples}[Theorem]{Examples}
\newtheorem{Remark}[Theorem]{Remark} }

\newcommand{\G}{\mathbb{G}}
\newcommand{\R}{\mathbb{R}}
\newcommand{\C}{\mathbb{C}}
\newcommand{\Z}{\mathbb{Z}}
\newcommand{\Q}{\mathbb{Q}}
\newcommand{\N}{\mathbb{N}}

\def \Gal{\operatorname{Gal}}

\def \DGal{\operatorname{DGal}}
\def \DAut{\operatorname{DAut}}
\def \DHom{\operatorname{DHom}}

\def \GL{\operatorname{GL}}
\def \SO{\operatorname{SO}}

\def \Id{\operatorname{Id}}
\def \grad{\operatorname{grad}}

\begin{document}
\allowdisplaybreaks

\newcommand{\arXivNumber}{2104.09548}

\renewcommand{\PaperNumber}{095}

\FirstPageHeading

\ShortArticleName{Real Liouvillian Extensions of Partial Differential Fields}

\ArticleName{Real Liouvillian Extensions\\ of Partial Differential Fields}

\Author{Teresa CRESPO~$^{\rm a}$, Zbigniew HAJTO~$^{\rm b}$ and Rouzbeh MOHSENI~$^{\rm b}$}

\AuthorNameForHeading{T.~Crespo, Z.~Hajto and R.~Mohseni}

\Address{$^{\rm a)}$~Departament de Matem\`atiques i Inform\`atica, Universitat de Barcelona,\\
\hphantom{$^{\rm a)}$}~Gran Via de les Corts Catalanes 585, 08007 Barcelona, Spain}
\EmailD{\href{mailto:teresa.crespo@ub.edu}{teresa.crespo@ub.edu}}
\URLaddressD{\url{http://www.ub.edu/tn/personal/crespo.php}}

\Address{$^{\rm b)}$~Faculty of Mathematics and Computer Science, Jagiellonian University,\\
\hphantom{$^{\rm a)}$}~ul.~\L ojasiewicza 6, 30-348 Krak\'{o}w, Poland }
\EmailD{\href{mailto:Zbigniew.Hajto@uj.edu.pl}{Zbigniew.Hajto@uj.edu.pl}, \href{mailto:Rouzbeh.Mohseni@doctoral.uj.edu.pl}{Rouzbeh.Mohseni@doctoral.uj.edu.pl}}

\ArticleDates{Received February 28, 2021, in final form October 25, 2021; Published online October 29, 2021}
\vspace{1mm}

\Abstract{In this paper, we establish Galois theory for partial differential systems defined over formally real differential fields with a real closed field of constants and over formally $p$-adic differential fields with a $p$-adically closed field of constants. For an integrable partial differential system defined over such a field, we prove that there exists a formally real (resp.~formally $p$-adic) Picard--Vessiot extension. Moreover, we obtain a uniqueness result for this Picard--Vessiot extension. We give an adequate definition of the Galois differential group and obtain a Galois fundamental theorem in this setting. We apply the obtained Galois correspondence to characterise formally real Liouvillian extensions of real partial differential fields with a real closed field of constants by means of split solvable linear algebraic groups. We present some examples of real dynamical systems and indicate some possibilities of further development of algebraic methods in real dynamical systems.}

\vspace{1mm}
\Keywords{real Liouvillan extension; real and $p$-adic Picard--Vessiot theory; split solvable algebraic group; gradient dynamical systems; integrability}

\vspace{1mm}
\Classification{12H05; 37J35; 12D15; 14P05}

\vspace{1mm}
\section{Introduction}
Galois theory for linear differential equations is the differential counterpart of classical Galois theory. The idea of Galois of characterising those polynomial equations solvable by radicals by means of the group of permutations of the roots which preserve the relations between them was paralleled in the work of Picard and Vessiot who characterised linear differential equations solvable by quadratures by means of the group of linear automorphisms of the vector space of solutions that preserve the differential relations between them. In a similar way as the work of Galois was later forma\-li\-zed by Artin, the one of Picard and Vessiot was forma\-li\-zed by using differential algebra, more precisely introducing the notion of a differential field~$K$ to be a~field endowed with a derivation. The constants of $K$ are defined as the elements on which the derivation vanishes and form a subfield $\mathcal{C}$ of $K$. A satisfactory Galois theory for linear differential equations defined over a differential field $K$ with algebraically closed field of constants~$\mathcal{C}$ was established by Kolchin, under the name of Picard--Vessiot theory (see~\cite{Kol0}). For~such a differential equation, Kolchin proved the existence and uniqueness up to $K$-differential isomorphism of a~Picard--Vessiot field, the analog of the splitting field in classical Galois theo\-ry. The differential Galois group is defined as the group of differential automorphisms of the Picard--Vessiot field which fix the base field $K$ and it has the structure of a linear algebraic group defined over the field of constants $\mathcal{C}$. The fundamental theorem of Picard--Vessiot theory establishes a~bijective correspondence between intermediate differential fields of the Picard--Vessiot extension and the closed subgroups of the differential Galois group. We note that the hypothesis that the field of constants $\mathcal{C}$ of $K$ is algebraically closed was essential in the work of Kolchin. His results can be extended to integrable partial differential systems defined over a partial differential field~$K$ with algebraically closed field of constants (see~\cite[Appendix~D]{PS}).\looseness=1

In~\cite{CHP}, T.~Crespo, Z.~Hajto and M.~van der Put show that the condition that the field of constants is algebraically closed may be relaxed. They consider formally real and formally $p$-adic fields, whose definition is given in Section \ref{prelim}.
For a homogeneous linear differential equation defined over a formally real (resp.\ formally $p$-adic) ordinary differential field $K$ with a real closed (resp.\ $p$-adically closed) field of constants $\mathcal{C}_K$, they prove the existence of a formally real (resp.\ formally $p$-adic) Picard--Vessiot extension. Moreover, they obtain a result of uniqueness of the formally real (resp.\ formally $p$-adic) Picard--Vessiot extension up to $K$-differential isomorphism. In model theoretic language, formally real (resp.\ formally $p$-adic) differential fields with a real closed (resp.\ $p$-adically closed) field of constants are instances of differential fields $K$ such that the field of constants $\mathcal{C}_K$ is existentially closed in $K$. Under this hypothesis, the existence of a~Picard--Vessiot extension is proved in~\cite[Theorem~2.2]{GGO}. In~\cite{KP} the results in~\cite{CHP} are generalised to the case when~$\mathcal{C}_K$ is existentially closed in~$K$, large and bounded.

The standard example of a real closed field is the field $\R$ of real numbers whereas the field of real rational functions in one or several variables is an example of a formally real field. Picard--Vessiot theory for formally real differential fields with real closed field of constants allows then to characterise some aspects of the behaviour of real functions. In~\cite{CH} T.~Crespo and Z.~Hajto characterised real Liouvillian extensions of ordinary differential fields in terms of real Picard--Vessiot theory. Their result answers an earlier question of A. Khovanskii, namely ``Is~it true that a necessary and sufficient condition for solvability of a real differential equation by real Liouville functions follows from real Picard--Vessiot theory?'', which originated in the theory developed in~\cite{GK}. For a generalisation and further development of this theory one can consult~\cite{Kh}.\looseness=1

In this paper, we establish Galois theory for partial differential systems over formally real or formally $p$-adic partial differential fields. We prove the existence of a formally real (resp.\ formally $p$-adic) Picard--Vessiot extension for a partial differential system defined over a formally real (resp.\ formally $p$-adic) partial differential field $K$ with a real closed (resp.\ $p$-adically closed) field of constants. We establish as well a uniqueness result of the formally real (resp.\ formally $p$-adic) Picard--Vessiot extension up to $K$-differential isomorphism. We give an adequate definition of the differential Galois group in this setting and prove a Galois correspondence theorem. Due to the fact that real Liouville functions over the field $\R$ of real numbers have very interesting topological properties (see~\cite{GK,Kh}), we restrict to formally real fields in our study of Liouvillianity questions. We characterise formally real Liouvillian extensions for formally real partial differential fields by means of the differential Galois group. Previously, we recall in Section \ref{prelim} the concepts of formally real field, real closed field, formally $p$-adic field and $p$-adically closed field as well as the main definitions and known results of Picard--Vessiot theory. It is worth noting that the algebraic characterisation of real Liouville functions allows us to expect an algebraic version of the theory developed by Gel'fond and Khovanskii in~\cite{GK,Kh}. Similar ideas already appear in Grothendieck's ``Esquisse d'un programme'' (see~\cite[p.~272]{GGA}), where he proposed to consider constant functions with values in the real closure $\overline{\Q}^r$ of the field of rational numbers $\Q$ in the study of integration of Pfaff systems. In the last section of the paper we present some questions related to the integrability of real dynamical systems showing the interest of this further development.\looseness=1


\section{Preliminaries}\label{prelim}

\subsection{Ordered fields, formally real fields}

In this section we recall the concept of formally real field and real closed field. More details on these topics may be found in~\cite{BCR,Pre}.

\begin{Definition} An \emph{ordering} on a field $k$ is a total order relation $\leq$ on $k$ satisfying, for any $x,y,z \in k$,
\begin{enumerate}\itemsep=0pt
\item[$(i)$] $x\leq y \Rightarrow x+z \leq y+z$,
\item[$(ii)$] $0\leq x$ and $0\leq y \Rightarrow 0\leq xy$.
\end{enumerate}

An \emph{ordered field} $(k,\leq)$ is a field $k$ equipped with an ordering $\leq$.
\end{Definition}

\begin{Remark}[{\cite[Example~1.1.2]{BCR}}]\label{ord}
Given an ordered field $(k,\leq)$, there is exactly one ordering on the field of rational functions $k(t)$, extending the one on $k$, such that $t$ is positive and smaller than any positive element in $k$. If $P(t)=a_n t^n+a_{n-1}t^{n-1}+\dots+a_m t^m \in k[t]$, with $a_m \neq 0$, then $P(t) >0 \Leftrightarrow a_m>0$. If $P(t)/Q(t) \in k(t)$, then $P(t)/Q(t)>0 \Leftrightarrow P(t)Q(t)>0$.
\end{Remark}

\begin{Theorem}[{\cite[Theorem~1.1.8]{BCR}}] Let $k$ be a field. The following properties are equivalent:
\begin{enumerate}\itemsep=0pt
\item[$(i)$] $k$ can be ordered,
\item[$(ii)$] $-1$ is not a sum of squares in $k$,
\item[$(iii)$] for $x_1,\dots,x_n \in k$, $\sum_{i=1}^n x_i^2=0 \Rightarrow x_1=\dots =x_n=0$.
\end{enumerate}
\end{Theorem}

\begin{Definition} A field satisfying the properties of the preceding theorem is called a \emph{formally real field}.
\end{Definition}

We note that a formally real field always has characteristic $0$.

\begin{Definition} If $k$ is a formally real field, a \emph{formally real extension} of $k$ is a field extension~$\ell/k$ such that $\ell$ is a formally real field. A \emph{real closed field} is a formally real field that has no nontrivial formally real algebraic extension.
\end{Definition}

\begin{Theorem}[{\cite[Theorem~1.2.2]{BCR}}]\label{realcl}
Let $k$ be a field. The following properties are equivalent:
\begin{enumerate}\itemsep=0pt
\item[$(i)$] $k$ is a real closed field,
\item[$(ii)$] $k$ has a unique ordering,
\item[$(iii)$] the ring $k[i]=k[X]/\big(X^2+1\big)$ is an algebraically closed field.
\end{enumerate}
\end{Theorem}

\begin{examples} $\Q$, $\R$ are formally real fields, $\R$ is a real closed field.
\end{examples}

\begin{Definition} A \emph{real algebraic closure} of an ordered field $(k,\leq)$ is an algebraic field extension~$k^r$ of~$k$ such that $k^r$ is a real closed field and the unique ordering of $k^r$ restricts to $\leq$ on~$k$.
\end{Definition}

\begin{Theorem}[{\cite[Theorem~3.10]{Pre}}] Any ordered field $(k,\leq)$ has a
unique $($up to $k$-isomorphism$)$ real algebraic closure.
\end{Theorem}

\subsection[Valued fields, p-adic fields]{Valued fields, $\boldsymbol p$-adic fields}

In this section we recall the concept of formally $p$-adic field and $p$-adically closed field. More details on these topics may be found in~\cite{PR}.

\begin{Definition} A \emph{valuation} of a field $k$ is a map
\begin{gather*}
v\colon\ k \rightarrow \Gamma \cup \{ \infty \},
\end{gather*}
where $\Gamma$ is a totally ordered abelian group, such that, for all $a$, $b$ in $k$,
\begin{enumerate}\itemsep=0pt
\item[$(1)$] $v(a)=\infty \Leftrightarrow a=0$,
\item[$(2)$] $v(ab)=v(a)+v(b)$,
\item[$(3)$] $v(a+b) \geq \min \{v(a),v(b)\}$, with equality if $v(a)\neq v(b)$.
\end{enumerate}
\end{Definition}

We recall that, for $v$ a valuation of a field $k$, the \emph{valuation ring} $\mathcal{O}$ is defined as $\mathcal{O}:=\{a \in k \colon v(a) \geq 0 \}$ and $\mathcal{O}$ has a unique maximal ideal $\frak{m}:= \{a \in k \colon v(a) > 0 \}$. The \emph{residue field} is then defined as the quotient $\mathcal{O}/\frak{m}$.

\begin{Definition}
Let $p$ be a prime number. A \emph{$p$-valuation} of a field $k$ of characteristic 0 is a~valuation $v$ of $k$
such that $v(p)$ is minimal positive in the value group and the residue field is isomorphic to $\Z/p\Z$.

A \emph{$p$-valued field} $(k,v)$ is a characteristic 0 field $k$ equipped with a $p$-valuation $v$. A \emph{formally $p$-adic field} is a characteristic 0 field which can be endowed with a $p$-valuation.
\end{Definition}

\begin{Remark}\label{val} If $(k,v)$ is a $p$-valued field, the $p$-valuation $v$ may be extended to a $p$-valuation of the field of rational functions $k(t)$ (see~\cite[Example~2.2]{PR}).
\end{Remark}

\begin{Definition} If $k$ is a formally $p$-adic field, a \emph{formally $p$-adic extension} of $k$ is a field extension $\ell/k$ such that $\ell$ is a formally $p$-adic field. A \emph{$p$-adically closed field} is a formally $p$-adic field that has no nontrivial formally $p$-adic algebraic extension.
\end{Definition}

\begin{Example} $\Q_p$ is a $p$-adically closed field.
\end{Example}

\begin{Remark} What we call ``$p$-adic'' is called ``$p$-adic of rank one'' in~\cite{PR}. The case of higher rank $p$-adic fields
can be treated in the same way.
\end{Remark}

\begin{Definition} A \emph{$p$-adic algebraic closure} of a $p$-valued field $(k,v)$ is an algebraic field extension $k^v$ of $k$ such that $k^v$ is a $p$-adically closed field and the valuation of $k$ extends to $k^v$.
\end{Definition}

\begin{Theorem}[{\cite[Corollary~3.11]{PR}}] Any $p$-valued field $(k,v)$ has a
$p$-adic algebraic closure. Let $\ell_1$ and $\ell_2$ be two $p$-adic closures of $k$. Then $\ell_1$ and $\ell_2$ are $k$-isomorphic if and only if
$k\cap \ell_1^n=k\cap \ell_2^n$ for all $n \in \N$ $($where $\ell_i^n:=\{a^n| a\in \ell_i\}$, $i=1,2)$.
\end{Theorem}

\subsection{Known results on Picard--Vessiot extensions}

Let $K$ be a field endowed with a set $\Delta=\{\partial_1,\dots,\partial_m \}$ of pairwise commuting derivations. Let~$\mathcal{C}_K$ denote the field of constants of $K$. In particular, if $m=1$, $K$ is an ordinary differential field. We consider a differential system
\begin{gather}\label{eq1}
\partial_j Y=A_j Y, \qquad 1\leq j \leq m,
\end{gather}
where $A_j$ is an $(r \times r)$-matrix with entries in $K$, $1\leq j \leq m$.

A \emph{solution} for the system \eqref{eq1} is a column vector $v$ in the vector space $L^r$, for $L$ some differential field extension of $K$, such that $\partial_j v= A_j v$, $1\leq j \leq m$. A \emph{fundamental matrix} for the system \eqref{eq1} is an $r\times r$ invertible matrix~$M$ with entries in some differential field extension~$L$ of~$K$ such that $\partial_j M= A_j M$, $1\leq j \leq m$. If $M \in \GL_r(L)$ is a fundamental matrix for~\eqref{eq1}, then the set of fundamental matrices for \eqref{eq1} with entries in $L$ consists in the matrices of the form~$MC$ for $C \in \GL_r(\mathcal{C}_L)$, for $\mathcal{C}_L$ the constant field of~$L$. We note that, if a fundamental matrix exists for~\eqref{eq1}, then the commutation of the derivatives implies that the matrices $A_j$ satisfy
\begin{gather}\label{eq2}
\partial_j A_i+A_i A_j=\partial_i A_j + A_j A_i, \qquad
1 \leq i,j\leq m.
\end{gather}
The relations (\ref{eq2}) are therefore a necessary condition for the existence of a common solution to the equations in the system \eqref{eq1}. We say that the differential system \eqref{eq1} is \emph{integrable} if the matrices $A_j$ satisfy (\ref{eq2}).

\begin{Definition}\label{PV}
\emph{A Picard--Vessiot extension} for the differential system \eqref{eq1}
is a field extension~$L$ of~$K$ such that:
\begin{enumerate}\itemsep=0pt
\item[$1.$] $L$ is equipped with a set of pairwise commuting derivations extending the ones in $\Delta$.
\item[$2.$] There exists a fundamental matrix $Z$ for \eqref{eq1} with entries in $L$.
\item[$3.$] $L$ is (as a field) generated over K by the entries of $Z$.
\item[$4.$] The field of constants of $L$ is the same as the field of constants of $K$.
\end{enumerate}
\end{Definition}

When $\mathcal{C}_K$ is algebraically closed, a classical result in differential Galois theory states that a~Picard--Vessiot field exists for the system \eqref{eq1} and it is unique up to a $K$-differential isomorphism (see, e.g.,~\cite[Theorems~5.6.5 and~5.6.9]{CHgsm} for the ordinary case and~\cite[Appendix~D]{PS} for the partial case).

In~\cite{CHP}, Crespo, Hajto and van der Put obtained the following theorem on existence and uniqueness of Picard--Vessiot fields for formally real ordinary differential fields with real closed field of constants and for formally $p$-adic ordinary differential fields with $p$-adically closed field of constants.

\begin{Theorem}[{\cite[Theorem~2]{CHP}}] \label{CHP}
Let $K$ be a formally real $($resp.\ formally $p$-adic$)$ ordinary differential field with real closed $($resp.\ $p$-adically closed$)$ field of constants. We consider a differential system
\begin{gather}\label{syso}
Y'=A Y,
\end{gather}
where $A$ is an $(r \times r)$-matrix with entries in $K$.
\begin{enumerate}\itemsep=0pt
\item[$1.$] {\bf \emph{Existence}.} There exists a formally real $($resp.\ formally $p$-adic$)$ Picard--Vessiot extension of~$K$ for~\eqref{syso}, i.e., a~Picard--Vessiot extension of $K$ for \eqref{syso} which is also a formally real $($resp.\ formally $p$-adic$)$ field extension of~$K$.
\item[$2.$] {\bf \emph{Unicity for the real case}.} Let $L_1$, $L_2$ denote two formally real Picard--Vessiot extensions of $K$ for \eqref{syso}. Suppose that $L_1$ and $L_2$ have orderings which induce the same ordering on~$K$. Then $L_1$ and $L_2$ are $K$-differentially isomorphic.
\item[$3.$] {\bf \emph{Unicity for the} $\boldsymbol p$-\emph{adic case}}. Let $L_1$, $L_2$ denote two formally $p$-adic
Picard--Vessiot extensions of $K$ for \eqref{syso}. Suppose that $L_1$ and $L_2$
have $p$-adic closures $L_1^+$ and $L_2^+$ such that the
$p$-valuations of $L_1^+$ and $L_2^+$ induce the same
$p$-valuation
on $K$ and such that $K\cap \big(L_1^+\big)^n=K\cap \big(L_2^+\big)^n$ for every integer $n\geq 2$. Then $L_1$ and $L_2$ are $K$-differentially isomorphic.
\end{enumerate}
\end{Theorem}

\section[Galois theory for partial differential systems over formally real or formally p-adic differential fields]
{Galois theory for partial differential systems \\over formally real or formally $\boldsymbol p$-adic differential fields}

In this section we consider an integrable partial differential system \eqref{eq1} defined over a formally real (resp.\ formally $p$-adic) differential field $K$ with real closed (resp.\ $p$-adically closed) field of constants $\mathcal{C}_K$. Let $\overline{\mathcal{C}_K}$ denote an algebraic closure of $\mathcal{C}_K$. The differential Galois group $G$ for~\eqref{eq1} is a $\overline{\mathcal{C}_K}/\mathcal{C}_K$-form of the differential Galois group $\overline{G}$ for~\eqref{eq1} considered as defined over $\overline{K}:=K\otimes_{\mathcal{C}_K} \overline{\mathcal{C}_K}$, which is a partial differential field with algebraically closed field of constants~$\overline{\mathcal{C}_K}$. However the differential Galois group $G$ for \eqref{eq1} over $K$ gives more information than $\overline{G}$ on the behaviour of the solutions to \eqref{eq1}.
In the case when $K$ is a formally real ordinary differential field and $\mathcal{C}_K$ is real closed, we obtained in~\cite{CH} that the property for a system $Y'=AY$ defined over $K$ to have solutions which are real
Liouville functions is characterised by the differential Galois group $G$. Therefore it is interesting to extend this result to partial differential fields.

\subsection{Picard--Vessiot extensions}

In this section, we will use the approach of Kolchin in~\cite{Kol1} to show how the Picard--Vessiot theory of formally real and formally $p$-adic partial differential fields can be deduced from the ordinary case. Let us note that Kolchin's definition of Picard--Vessiot extension for partial differential fields is different from the one given in Definition \ref{PV}. In~\cite[Theorem 1]{CHJac} the equivalence of both definitions is proved.

\begin{Remark}\label{D} Following Kolchin~\cite{Kol1}, to a partial differential field $k$ with pairwise commu\-ting derivations
$\partial_1, \dots ,\partial_m$ we associate an ordinary differential field $k_D:=k\langle u_1,\dots,u_m \rangle$
with $u_1,\dots,u_m$ independent differential indeterminates, endowed with the derivation $D:=u_1\partial_1+\dots + u_m\partial_m$.
As remarked by Kolchin, $k_D$ and $k$ have the same field of constants.
Let us observe that, as a field, $k_D$ is a purely transcendental extension of $k$. Therefore, if $k$ is a formally real field (resp.\ a formally $p$-adic field), then using induction and Remark \ref{ord} (resp.\ Remark \ref{val}), the ordering (resp.\ the $p$-valuation) of $k$ may be extended to $k_D$, i.e., $k_D$ is a formally real (resp.\ formally $p$-adic) field.
\end{Remark}

We consider now a system of the form \eqref{eq1} defined over a formally real (resp.\ formally $p$-adic) partial differential field. We shall prove the following theorems on the existence and uniqueness of Picard--Vessiot extensions in this setting.

\begin{Theorem} \quad
\begin{enumerate}\itemsep=0pt
\item[$1.$] Let us suppose that $K$ is a formally real partial differential field with real closed field of constants $\mathcal{C}_K$. Then for an integrable differential system \eqref{eq1} defined over $K$, there exists a~formally real differential field $L$ with field of constants $\mathcal{C}_K$ such that $L=K(\{y_{ij}\}_{1\leq i,j\leq r})$, for~$(y_{ij})_{1\leq i,j\leq r}\in \GL_r(L)$ a fundamental matrix for \eqref{eq1}, i.e., $L/K$ is a~Picard--Vessiot extension for \eqref{eq1}.
\item[$2.$] Let us suppose that $K$ is a formally $p$-adic partial differential field with $p$-adically closed field of constants $\mathcal{C}_K$. Then for an integrable differential system \eqref{eq1} defined over $K$, there exists a formally $p$-adic differential field $L$ with field of constants $\mathcal{C}_K$ such that $L=K(\{y_{ij}\}_{1\leq i,j\leq r})$, for $(y_{ij})_{1\leq i,j\leq r}\in \GL_r(L)$ a fundamental matrix for \eqref{eq1}, i.e., $L/K$ is a~Picard--Vessiot extension for \eqref{eq1}.
\end{enumerate}
\end{Theorem}

\begin{proof} First we consider an auxiliary differential system
\begin{gather}\label{aux}
DY=A_D Y,
\end{gather}
where $D=u_1\partial_1+\dots +u_m \partial_m$ and $A_D=u_1A_1+\dots +u_mA_m \in M_{r\times r} (K_D)$, with $u_1,\dots,u_m$ independent differential indeterminates. Since $(K_D,D)$ is a formally real (resp.\ formally $p$-adic) ordinary differential field, with real closed (resp.\ $p$-adically closed) field of constants $\mathcal{C}_K$, by Theo\-rem~\ref{CHP}, there exists a formally real (resp.\ formally $p$-adic) Picard--Vessiot extension~$L_D/K_D$ for the system (\ref{aux}). Let $Z$ denote a fundamental matrix for (\ref{aux}) such that~$L_D$ is generated over~$K_D$ by the entries of~$Z$. Let $\overline{\mathcal{C}_K}$ denote an algebraic closure of $\mathcal{C}_K$. The derivation~$D$ extends uniquely to $\overline{K}_D:=K_D\otimes_{\mathcal{C}_K} \overline{\mathcal{C}_K}$ and $\overline{L}_D:=L_D\otimes_{\mathcal{C}_K} \overline{\mathcal{C}_K}$, with field of constants $\overline{\mathcal{C}_K}$. Since the extensions $\overline{\mathcal{C}_K}/\mathcal{C}_K$ and $K_D/\mathcal{C}_K$ are linearly disjoint, $\overline{K}_D$ and $\overline{L}_D$ are fields and $\overline{L}_D$ is generated over $\overline{K}_D$ by the entries of $Z$. Hence $\overline{L}_D/\overline{K}_D$ is a~Picard--Vessiot extension for (\ref{aux}).

On the other hand, since $\overline{K}:=K\otimes_{\mathcal{C}_K} \overline{\mathcal{C}_K}$ is a partial differential field with algebraically closed field of constants $\overline{\mathcal{C}_K}$, we have a~Picard--Vessiot extension $\mathcal{L}/\overline{K}$ for the system \eqref{eq1} (see~\cite[Appendix~D]{PS}).

The field extensions we have considered up to now are shown in the following diagram:
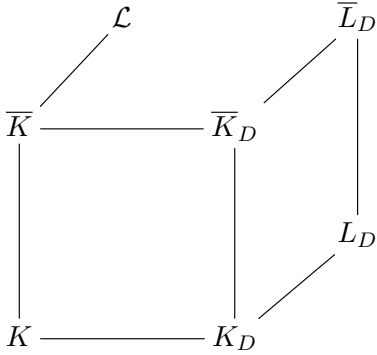
\vspace{1ex}
\begin{gather*}
\xymatrix{ & \mathcal{L}
&& \overline{L}_D \ar@{-}[dd]\\
\overline{K}\ar@{-}[ur]\ar@{-}[rr]\ar@{-}[dd]
&& \overline{K}_D\ar@{-}[ur]\ar@{-}[dd]\\
& & & L_D\\
K\ar@{-}[rr]
&& K_D\ar@{-}[ur]}
\end{gather*}

We define $\mathcal{L}_D:=\mathcal{L}\langle u_1,\dots,u_m \rangle$ and endow it with the derivation $D=u_1\partial_1+\dots +u_m \partial_m$. By~\cite[Theorem~1]{Kol1}, $\mathcal{L}_D$ is a~Picard--Vessiot extension of $\overline{K}_D$ for (\ref{aux}). By the uniqueness of the Picard--Vessiot extension in the ordinary case, $\mathcal{L}_D$ is isomorphic to $\overline{L}_D$ and we can assume $\mathcal{L} \subset \overline{L}_D$.

The Galois group $G:=\Gal(\overline{\mathcal{C}_K}/\mathcal{C}_K)$ acts on the field $\overline{L}_D=L_D\otimes_{\mathcal{C}_K} \overline{\mathcal{C}_K}$ by acting on the second factor. We note that this action commutes with the derivation $D$ and is continuous. Let~$V$ be the $\overline{\mathcal{C}_K}$-vector space of solutions to \eqref{eq1} contained in $\mathcal{L}$. By restricting the action of $G$ on $\overline{L}_D$ to $V$ we obtain a continuous semi-linear action. Let $V^G:=\{ v \in V \mid \sigma(v)=v, \, \forall \sigma \in G \}$. Clearly $V^G$ is a $\mathcal{C}_K$-subspace of the $\mathcal{C}_K$-vector space $V$. We want to show that the $\mathcal{C}_K$-vector space $V^G$ has dimension equal to the dimension of $V$ over $\overline{\mathcal{C}_K}$.

Let $v_1,\dots,v_d$ be a $\mathcal{C}_K$-basis of $V^G$. We shall prove first that $v_1,\dots,v_d$ are linearly independent over $\overline{\mathcal{C}_K}$. Assume otherwise and let $a_1v_1+\dots +a_dv_d=0$, with $a_i \in \overline{\mathcal{C}_K}$, $1 \leq i \leq d$ be a dependence relation with a minimal number of non-zero coefficients. We may assume $a_1 \neq 0$ and, by scaling, $a_1=1$. By applying $\sigma \in G$ to the dependence relation and subtracting the obtained relation from the first one, we obtain from the minimality condition $\sigma(a_i)-a_i=0$ for $2\leq i \leq d$. Since $\sigma$ is any element in $G$, we obtain that the coefficients $a_i$ belong to $\mathcal{C}_K$. This gives a contradiction, hence $v_1,\dots,v_d$ are $\overline{\mathcal{C}_K}$-linearly independent.

In order to prove that $v_1,\dots,v_d$ generate $V$ over $\overline{\mathcal{C}_K}$, we shall see that a vector $v \in V$ can be written as a linear combination of vectors in $V^G$, with coefficients in $\overline{\mathcal{C}_K}$. In the real case, by Theorem \ref{realcl}, we have $\overline{\mathcal{C}_K}=\mathcal{C}_K({\rm i})$, where ${\rm i}=\sqrt{-1}$, and $G=\{ \Id,c \}$ with $c$ defined by $c({\rm i})=-{\rm i}$. We may then write $v=(1/2)(v+c(v))+{\rm i} \, (1/2{\rm i})(v-c(v))$ and $(1/2)(v+c(v)),(1/2{\rm i})(v-c(v))$ belong to $V^G$. In the $p$-adic case, $G$ is a pro-finite group. For $v \in V \subset \overline{L}_D=L_D\otimes_{\mathcal{C}_K} \overline{\mathcal{C}_K}$, there exists a finite extension $\widetilde{\mathcal{C}}/\mathcal{C}$ such that $v \in L_D\otimes_{\mathcal{C}_K} \widetilde{\mathcal{C}}$ and we may assume $\widetilde{\mathcal{C}}/\mathcal{C}_K$ Galois. Let $n:=\big[\widetilde{\mathcal{C}}\colon \mathcal{C}_K\big]$ and $\widetilde{G}:=\Gal\big(\widetilde{\mathcal{C}}/\mathcal{C}_K\big)=\{\sigma_1,\dots,\sigma_n\}$. By the linear independence of characters and the equality of dimensions, the map
\begin{gather*}
\widetilde{\mathcal{C}}\otimes_{\mathcal{C}_K} \widetilde{\mathcal{C}} \rightarrow \widetilde{\mathcal{C}}^{n}, \qquad
x\otimes y \mapsto (x\sigma_1(y),\dots,x\sigma_n(y))
\end{gather*}
is an isomorphism. We may then find elements $x_1,\dots,x_n,y_1,\dots,y_n$ in $\widetilde{\mathcal{C}}$ such that
\begin{gather*}
\sum_{i=1}^n x_i y_i =1, \qquad
\sum_{i=1}^n x_i \sigma(y_i) =0,\qquad \text{for}\quad \sigma \in \widetilde{G},\quad \sigma \neq \Id.
\end{gather*}
We have then $v=\sum_{\sigma \in \widetilde{G}} \sum_{i=1}^n x_i \sigma(y_i) \sigma(v)= \sum_{i=1}^n x_i \big(\sum_{\sigma \in \widetilde{G}} \sigma(y_i v)\big)$ and $\sum_{\sigma \in \widetilde{G}} \sigma(y_i v) \in V^G$, for $i=1,\dots, n$.

We have then obtained that the $\mathcal{C}_K$-vector space $V^G$ has dimension $r=\dim_{\overline{\mathcal{C}_K}} V$. Now $V^G$ is clearly a vector space of solutions to \eqref{eq1}. Let us define $L:=K\big\langle V^G \big\rangle$ the subfield of $\mathcal{L}$ generated by the elements in $V^G$. By construction, $L$ is a~Picard--Vessiot extension of~$K$ for~\eqref{eq1}. Since $L\subset L_D$, the field $L$ is a formally real (resp.\ formally $p$-adic) field.

The complete diagram of field extensions is as follows:
\begin{gather*}
\begin{split}
&\xymatrix{ & \mathcal{L}\ar@{-}[rr]\ar@{-}'[d][dd]
&& \mathcal{L}_D \simeq \overline{L}_D\ar@{-}[dd]\\
\overline{K}\ar@{-}[ur]\ar@{-}[rr]\ar@{-}[dd]
&& \overline{K}_D\ar@{-}[ur]\ar@{-}[dd]\\
& L\ar@{-}'[r][rr] & & L_D\\
K\ar@{-}[rr]\ar@{-}[ur]
&& K_D\ar@{-}[ur]}\end{split}
\tag*{\qed}
\end{gather*}
\renewcommand{\qed}{}
\end{proof}

\begin{Theorem}\label{realuni} Let $L_1$, $L_2$ denote two formally real Picard--Vessiot extensions of $K$ for \eqref{eq1}. Suppose that $L_1$ and $L_2$ have orderings which induce the same ordering on $K$. Then $L_1$ and $L_2$ are $K$-differentially isomorphic.
\end{Theorem}

\begin{proof}
 We have that $(L_1)_D$ and $(L_2)_D$ are two formally real Picard--Vessiot extensions of $K_D$ for the system (\ref{aux}). We extend the ordering from $K$ to $K_D$ and from $L_i$ to $(L_i)_D$, $i=1,2$, in such a way that the ordering of $(L_i)_D$ restricts to the ordering in $K_D$, $i=1,2$ (see Remark~\ref{D}). By Theorem \ref{CHP}, we have a differential $K_D$-isomorphism
\begin{gather*}
\varphi\colon\ L_1\langle u_1,\dots,u_m \rangle \rightarrow L_2\langle u_1,\dots,u_m \rangle.
\end{gather*}

If $L_1=K(y_{ij})$ and $L_2=K\big(\widetilde{y}_{ij}\big)$ for fundamental matrices $(y_{ij})$ and $\big(\widetilde{y}_{ij}\big)$ of the system~\eqref{eq1}, then $(\varphi(y_{ij}))_{1\leq i,j\leq r}$ is a fundamental matrix of \eqref{eq1} which gives $(\varphi(y_{ij}))=\big(\widetilde{y}_{ij}\big) C$, with $C \in M_{r\times r} (\mathcal{C}_K)$. Therefore $\varphi$ restricts to a differential $K$-isomorphism
\begin{gather*}
\varphi_{|L_1}\colon\ L_1 \rightarrow L_2. \tag*{\qed}
\end{gather*}
\renewcommand{\qed}{}
\end{proof}

\begin{Theorem} Let $L_1$, $L_2$ denote two formally $p$-adic Picard--Vessiot extensions of $K$ for \eqref{eq1}. Suppose that $L_1$ and $L_2$
have $K$-isomorphic $p$-adic closures $L_1^+$ and $L_2^+$. Then $L_1$ and $L_2$ are $K$-differentially isomorphic.
\end{Theorem}

\begin{proof} As $u_1,\dots,u_m$ are independent differential indeterminates, the $K$-isomorphism from $L_1^+$ to $L_2^+$ may be extended to a field isomorphism from $\big(L_1^+\big)_D$ to $\big(L_2^+\big)_D$ over~$K_D$. Now we may extend the $p$-adic valuation to $\big(L_1^+\big)_D$ and define a $p$-valuation on $\big(L_2^+\big)_D$ by transferring to it the valuation of $\big(L_1^+\big)_D$ by means of the isomorphism.
Next we choose a $p$-adic closure $M_1$ of~$\big(L_1^+\big)_D$. We can extend the isomorphism from $\big(L_1^+\big)_D$ to $\big(L_2^+\big)_D$ to an embedding of $M_1$ into an algebraic closure of $\big(L_2^+\big)_D$. Its image, equipped with the valuation induced by the valuation of $\big(L_1^+\big)_D$ via the embedding is a $p$-adic closure $M_2$ of $\big(L_2^+\big)_D$. Now $M_1$ and $M_2$ are isomorphic as valued fields over $K_D$ and are $p$-adic closures of $(L_1)_D$ and $(L_2)_D$. By Theorem~\ref{CHP}, we have a differential $K_D$-isomorphism
\begin{gather*}
\varphi\colon\ L_1\langle u_1,\dots,u_m \rangle \rightarrow L_2\langle u_1,\dots,u_m \rangle.
\end{gather*}

The rest of the proof is like for Theorem \ref{realuni}.
\end{proof}

\subsection{Galois correspondence}

In the case when the field of constants of the differential field $K$ is algebraically closed, the dif\-fe\-ren\-tial Galois group of a~Picard--Vessiot extension $L/K$ is defined as the group of $K$-differential automorphisms of $L$. The next example illustrates that, when the field of constants is not algebraically closed, the group of $K$-differential automorphisms of $L$ may be ``too small'' to obtain a satisfactory differential Galois theory.

\begin{Example} Let $K:=\R(t_1,t_2)$ be the field of rational functions in the variables $t_1$, $t_2$ over the field $\R$ of real numbers endowed with the usual derivations $\partial_1:= \partial /\partial t_1$ and $\partial_2:= \partial /\partial t_2$. Clearly the field of constants of $K$ is $\R$. We consider the differential system
\begin{gather}\label{ex}
\partial_1 Y= t_2 \, Y, \qquad \partial_2 Y= t_1 \, Y,
\end{gather}

\noindent
defined over $K$. A solution to the system (\ref{ex}) is ${\rm e}^{t_1t_2}$, hence $L:=K\big({\rm e}^{t_1t_2}\big)$ is a formally real Picard--Vessiot extension of $K$ for (\ref{ex}). If $\sigma$ is a $K$-differential automorphism of $L$, we have $\sigma\big({\rm e}^{t_1t_2}\big)=\lambda {\rm e}^{t_1t_2}$, with $\lambda \in \R\setminus \{0\}$ and composing two such automorphisms amounts to multiplying the corresponding factors $\lambda$. Hence the group $\DAut_K L$ of $K$-differential automorphisms of $L$ is isomorphic to the multiplicative group of $\R$. We consider now the intermediate field $F:=K\big({\rm e}^{3t_1t_2}\big)$. An $F$-differential automorphism $\tau$ of $L$ is given by $\tau\big({\rm e}^{t_1t_2}\big)=\lambda {\rm e}^{t_1t_2}$, with $\lambda^3=1$, since $\tau$ must fix ${\rm e}^{3t_1t_2}$. Hence the group of $F$-differential automorphisms of $L$ is trivial. The subgroups of $\DAut_K L$ corresponding to $F$ and $L$ by the Galois correspondence will then be equal whereas $F$ and $L$ are not equal.
\end{Example}

Let $K$ be a formally real (resp.\ formally $p$-adic) partial differential field with real closed (resp.\ $p$-adically closed)
field of constants $\mathcal{C}_K$ and let $L/K$ be a Picard--Vessiot extension. Let $\overline{\mathcal{C}_K}$ denote an algebraic closure of $\mathcal{C}_K$ and consider the fields $\overline{K}:=K\otimes_{\mathcal{C}_K} \overline{\mathcal{C}_K}$ and $\overline{L}:=L\otimes_{\mathcal{C}_K} \overline{\mathcal{C}_K}$. As
in the ordinary case (see~\cite{CHSA}), we shall consider the set
$\DHom_K\big(L,\overline{L}\big)$ of $K$-differential morphisms from $L$ into
$\overline{L}$ and transfer the group structure from $\DAut_{\overline{K}} \overline{L}$ to
$\DHom_K\big(L,\overline{L}\big)$ by means of the bijection
\begin{align*}
\DAut_{\overline{K}} \overline{L} &\ \rightarrow \ \DHom_K\big(L,\overline{L}\big),
\\
\tau &\ \mapsto \ \tau_{|L}.
\end{align*}

\begin{Definition} For the Picard--Vessiot extension $L/K$ we define the differential Galois group as the set $\DHom_K\big(L,\overline{L}\big)$ endowed with the group structure given above and denote it by $\DGal(L|K)$.
\end{Definition}

\begin{Proposition}\label{groups}
The map $\DGal(L_D|K_D) \rightarrow \DGal(L|K)$ which to each
$K_D$-differential morphism from $L_D$ into $\overline{L}_D=L_D\otimes_{\mathcal{C}_K} \overline{\mathcal{C}_K}$
assigns its restriction to $L$ is an isomorphism of groups.
\end{Proposition}

\begin{proof} Since $L/K$ is a~Picard--Vessiot extension, we have that $L=K\big( \{y_{ij}\}_{1\leq i,j \leq r} \big)$, for $(y_{ij})_{1\leq i,j \leq r}$ a fundamental matrix for some differential system $S$ defined over $K$. If $\sigma$ is a~$K_D$-differential morphism from $L_D$ into
$\overline{L}_D$, then $\sigma((y_{ij}))$ is a fundamental matrix for $S$. Therefore there exists an invertible matrix $C$ with entries in the field of constants $\overline{\mathcal{C}_K}$ of $\overline{L}_D$ such that $\sigma((y_{ij}))=(y_{ij})C$, hence $\sigma((y_{ij}))$ has entries in $\overline{L}$. By restricting $\sigma$ to $L$, we obtain then a~$K$-morphism from $L$ to $\overline{L}$. For an element $x$ in $L$, we have $Dx=\sum_{i=1}^m u_i \partial_i x$ and $\sigma (Dx)=\sum_{i=1}^m u_i \sigma(\partial_i x)$, since $u_i \in K_D$, $1\leq i \leq m$. On the other hand $D(\sigma(x))=\sum_{i=1}^m u_i \partial_i (\sigma (x))$. Since the elements $u_i$ are algebraically independent, we obtain that $\sigma_{|L}$ commutes with $\partial_i$, $1 \leq i \leq m$. Hence $\sigma_{|L}$ is a $K$-differential morphism from $L$ to $\overline{L}$. Reciprocally, a $K$-differential morphism from $L$ to $\overline{L}$ may be extended to a $K_D$-differential morphism from $L_D$ to $\overline{L_D}$.
\end{proof}

\begin{Theorem}
The differential Galois group $\DGal(L|K)$ is a $\mathcal{C}_K$-defined closed subgroup of some general linear group defined over $\overline{\mathcal{C}_K}$, i.e., a linear algebraic group defined over $\mathcal{C}_K$. The $\mathcal{C}_K$-valued points of $\DGal(L|K)$ correspond to the $K$-differential automorphisms of $L$.
\end{Theorem}

\begin{proof} The first assertion follows from Proposition \ref{groups} and~\cite[Proposition~1]{CHSA}. The second one is clear from the definition of $\DGal(L|K)$.
\end{proof}

For a closed subgroup $H$ of $\DGal(L|K)$, $L^H$ is a partial
differential subfield of $L$ containing~$K$. If $E$ is an
intermediate partial differential field, i.e., $K\subset E \subset
L$, then $L/E$ is a~Picard--Vessiot extension and $\DGal(L|E)$
is a $\mathcal{C}_K$-defined closed subgroup of $\DGal(L|K)$. We obtain a Galois correspondence theorem.

\begin{Theorem} Let $K$ be a formally real $($resp.\ formally $p$-adic$)$ partial differential field with real closed $($resp.\ $p$-adically closed$)$ field of constants $\mathcal{C}_K$, let $L/K$ be a~Picard--Vessiot extension and $\DGal(L|K)$ be its differential Galois group.
\begin{enumerate}\itemsep=0pt
\item[$1.$]
The correspondences
\begin{gather*}
H\mapsto L^H , \qquad E\mapsto \DGal(L|E)
\end{gather*}
define inclusion inverting mutually inverse bijective
maps between the set of $\mathcal{C}_K$-defined closed subgroups $H$ of $\DGal(L|K)$
and the set of partial differential fields $E$ with \mbox{$K\!\subset\! E \!\subset\!
L$}.

\item[$2.$] The intermediate partial differential field $E$ is a~Picard--Vessiot extension of $K$ if and only if the subgroup $\DGal(L|E)$ is normal in $\DGal(L|K)$. In this case, the restriction morphism
\begin{align*}
\DGal(L|K)&\ \rightarrow\ \DGal(E|K),
\\
\sigma &\ \mapsto \ \sigma_{|E}
\end{align*}
induces an isomorphism
\begin{gather*}
\DGal(L|K)/\DGal(L|E) \simeq \DGal(E|K).
\end{gather*}
\end{enumerate}
\end{Theorem}

\begin{proof}
The proof follows the same steps as the one of~\cite[Theorem~1]{CHSA}. In the present case, we use the Galois correspondence theorem for partial differential fields with an algebraically closed field of constants (see~\cite[Appendix~D]{PS}).
\end{proof}

\subsection{Liouvillian extensions}\label{sLiou}

In this section we restrict to formally real partial differential fields. We note that real Liouville functions may be characterised by topological properties (see~\cite{GK}) and it is therefore interesting to characterise real Liouville solutions to differential systems defined over formally real fields by means of the differential Galois group.

\begin{Definition}\ Let $K$ be a field endowed with pairwise commuting derivations $\partial_1,\dots,\partial_m$. Let $L/K$ be a partial differential field extension, $\alpha$ an element in $L$. We say that $\alpha$ is
\begin{enumerate}\itemsep=0pt
\item[$-$] \emph{an integral over $K$} if $\partial_k \alpha = a_k \in K$, $1 \leq k \leq m$, and $a_k$ is not a derivative in $K$ for all $k$;
\item[$-$] \emph{the exponential of an integral over $K$} if $\partial_k \alpha/\alpha \in K\setminus \{0\}$, $1 \leq k \leq m$.
\end{enumerate}
\end{Definition}

\begin{Remark}\label{intexp}
If $\partial_k \alpha = a_k \in K$, for $1 \leq k \leq m$, then $D \alpha = \sum_{k=1}^m u_k a_k \in K_D$ and if $a_k$ is not a derivative in $K$ for all $k$, then $\sum_{k=1}^m u_k a_k$ is not a derivative in $K_D$. So, if $\alpha$ is an integral over $K$, it is also an integral over $K_D$. Analogously, if $\partial_k \alpha/\alpha \in K\setminus \{0\}$, for $1 \leq k \leq m$, then $D\alpha /\alpha \in K_D\setminus \{0\}$. So, if $\alpha$ is the exponential of an integral over $K$, it is also the exponential of an integral over $K_D$.
\end{Remark}

 Let now $K$ be a formally real field with real closed field of constants $\mathcal{C}_K$. From Proposition~\ref{groups} and the corresponding results in the ordinary case \cite[Examples~7 and~8]{CH}, we obtain that, if $\alpha$ is an integral over $K$, then $K\langle \alpha \rangle/K$ is a real Picard--Vessiot extension and its differential Galois group $\DGal(K\langle \alpha
\rangle|K)$ is isomorphic to the additive group $\G_a$; and, if $\alpha$ is the
exponential of an integral and the field $K\langle \alpha \rangle$ is real and with field of constants equal to $\mathcal{C}_K$, then $K\langle \alpha \rangle/K$ is a~Picard--Vessiot extension and $\DGal(K\langle \alpha
\rangle|K)$ is isomorphic to the multiplicative group $\G_m$, or~a~finite subgroup of it.

\begin{Definition} A partial differential field extension $L/K$ is called a \emph{Liouvillian extension} (resp.\ a \emph{generalised Liouvillian extension}) if there exists a chain of intermediate partial differential fields $K=F_1 \subset F_2 \subset \dots \subset F_n =L$ such that $F_{i+1}=F_i(\alpha_i)$, where $\alpha_i$ is either an integral or the exponential of an integral over $F_i$ (resp.\ or $\alpha_i$ is algebraic over $F_i$), $1\leq i \leq n-1$.
\end{Definition}

\begin{Lemma}\label{Liou}
 If $L/K$ is a $($generalised$)$ Liouvillian extension of partial differential fields then the ordinary differential field extension $L_D/K_D$ is a $($generalised$)$ Liouvillian extension.
\end{Lemma}

\begin{proof}
 Let $K=F_1 \subset F_2 \subset \dots \subset F_n =L$ be a chain of intermediate partial differential fields such that $F_{i+1}=F_i(\alpha_i)$.
 We consider the chain of ordinary differential fields $K_D \subset (F_2)_D \subset \dots \subset L_D$. We have $(F_{i+1})_D=F_{i+1}\langle u_1,\dots,u_m \rangle=F_i(\alpha_i)\langle u_1,\dots,u_m \rangle=F_i\langle u_1,\dots,u_m \rangle(\alpha_i)=(F_i)_D(\alpha_i)$. Now, by Remark \ref{intexp}, if $\alpha_i$ is an integral or the exponential of an integral or algebraic over $F_i$, then $\alpha_i$ is an integral or the exponential of an integral or algebraic over $(F_i)_D$, respectively.
\end{proof}

\begin{Definition} Let $G$ be a connected solvable linear algebraic group defined over a field $\mathcal{C}$. We say that $G$ is $\mathcal{C}$-split if it has a composition series $G=G_1 \supset G_2 \supset \cdots \supset G_s = 1$ consisting of connected $\mathcal{C}$-defined closed subgroups such that $G_i/G_{i+1}$ is $\mathcal{C}$-isomorphic to $\G_a$ or $\G_m$, $1\leq i <s$.
\end{Definition}

From the results obtained in the ordinary case \cite[Section~3, Theorems~17 and~18]{CH}, we obtain the characterisation of real
Liouvillian extensions of real partial differential fields.

\begin{Theorem}\label{genL}
Let $K$ be a real partial differential field with real closed field of constants $\mathcal{C}_K$, $L/K$ be a formally real Picard--Vessiot extension, $\DGal(L|K)$ be its differential Galois group. The following conditions are equivalent:
\begin{enumerate}\itemsep=0pt
\item[$1.$] $L/K$ is a generalised Liouvillian extension.
\item[$2.$] $L$ is contained in a generalised Liouvillian extension $M$ of $K$, such that the field of constants of $M$ is $\mathcal{C}_K$ and $M$ is a formally real field.
\item[$3.$] The identity component of $\DGal(L|K)$ is solvable and $\mathcal{C}_K$-split.
\end{enumerate}
\end{Theorem}

\begin{proof} Clearly (1) implies (2). By Lemma \ref{Liou}, (2) implies that $L_D$ is contained in a formally real generalised Liouvillian extension of $K_D$, not adding constants to $K_D$. Then by Proposition~\ref{groups} and~\cite[Theorem~18]{CH}, we obtain (3).

We assume now that the identity component $G^0$ of $G:=\DGal(L|K)$ is solvable and $\mathcal{C}_K$-split and let $L^0=L^{G^0}$. Since $\big[G:G^0\big]$ is finite, $L^0/K$ is a finite extension. By~\cite[Chapter~V, Theorem~15.4]{B}, in $G^0$ is triangularizable by a matrix with entries in~$\mathcal{C}_K$. We may then assume that there exist elements
$v_1,\dots,v_r \in L$ such that $L=L^0\langle v_1,\dots,v_r \rangle$ and for every $\sigma \in G^0$ we~have
\begin{gather}\label{eq:liou}
\sigma v_j = a_{1j} v_1 + \dots + a_{j-1,j} v_{j-1}+a_{jj} v_j , \qquad j=1,\dots,r,
\end{gather}
with $a_{ij}$ constants in $L(i)$ (depending on $\sigma$). The first equality is $\sigma v_1=a_{11} v_1$ which implies $\sigma(\partial_k v_1/v_1)=\partial_k v_1/v_1$, for all $\sigma \in G^0$ and all $k=1,\dots,m$, hence $v_1$ is the exponential of an integral over $L^0$. Now dividing the equations in (\ref{eq:liou}) for $j=2$ to $r$ by the first equation and applying $\partial_k$, we obtain
\begin{gather*}
\sigma \partial_k \bigg(\dfrac {v_j}{v_1}\bigg)
= \dfrac {a_{2j}}{a_{11}} \partial_k \bigg(\dfrac{v_2}{v_1} \bigg)+ \dots + \dfrac {a_{j-1,j}}{a_{11}} \partial_k \bigg(\dfrac{ v_{j-1}}{v_1} \bigg)+
\dfrac{a_{jj}}{a_{11}} \partial_k \bigg(\dfrac{ v_j}{v_1} \bigg), \qquad j=2,\dots,r,
\end{gather*}
for $k=1,\dots,m$. By induction hypothesis on $r$, we obtain that the extension
\begin{gather*}
L^1:=L^0\langle v_1 \rangle \bigg\langle \bigg\{\partial_k \bigg(\dfrac {v_j}{v_1}\bigg)\bigg\}_{k=1,\dots,m; j=2,\dots,r}\bigg\rangle\Big/L^0
\end{gather*}
is a Liouvillian extension. Now the elements $v_j/v_1$ are integral over $L^1$ and we obtain that $L/L^0$ is a Liouvillian extension and $L/K$ is a generalised Liouvillian extension.
\end{proof}

\begin{Remark}
From Lemma \ref{Liou}, Proposition \ref{groups} and Theorem \ref{genL}, we obtain that $L/K$ is a formally real (generalised) Liouvillian extension of partial differential fields if and only if the ordinary differential field extension $L_D/K_D$ is a formally real (generalised) Liouvillian extension.
\end{Remark}

\begin{Example} Let $K:=\R(t_1,t_2)$ be the field of rational functions in the variables $t_1$, $t_2$ over the field $\R$ of real numbers endowed with the usual derivations $\partial_1:= \partial /\partial t_1$ and $\partial_2:= \partial /\partial t_2$. The field of constants of $K$ is $\R$. We consider the differential system
\begin{gather}\label{ex2}
\partial_1 Y= \begin{pmatrix} 1/t_1 & 0 \\ 0 & 1/t_1 \end{pmatrix} Y,\qquad
\partial_2 Y= \begin{pmatrix} 0&1 \\ -1 &0 \end{pmatrix} Y,
\end{gather}
defined over $K$. A fundamental matrix for the system (\ref{ex2}) is
\begin{gather*}
\begin{pmatrix} t_1 \sin t_2 & -t_1 \cos t_2 \\ t_1 \cos t_2 & t_1 \sin t_2 \end{pmatrix}\!,
\end{gather*}
hence $L:=K\langle t_1 \sin t_2,t_1 \cos t_2 \rangle=K\langle \sin t_2, \cos t_2 \rangle$ is a~Picard--Vessiot extension of $K$ for~\eqref{ex2}.
An element $\sigma$ in the Galois group $\DGal(L/K)$ is determined by $\sigma(\sin t_2)$ and $\sigma(\cos t_2)$. If~$\sigma(\sin t_2)\allowbreak=a\sin t_2 +b \cos t_2$, then $\sigma(\cos t_2)=\sigma(\partial_2(\sin t_2))=\partial_2(\sigma(\sin t_2))=a\cos t_2 - b \sin t_2$. Moreover $(\sin t_2)^2+ (\cos t_2)^2=1$ implies $(\sigma(\sin t_2))^2+ (\sigma(\cos t_2))^2=1$, which gives $a^2+b^2=1$. We obtain then $\DGal(L/K) \simeq \SO_2$.

We consider now the differential system
\begin{gather}\label{ex3}
\partial_1 Y= \begin{pmatrix} 1/t_1 & 0 \\ 0 & 1/t_1 \end{pmatrix} Y,\qquad
\partial_2 Y= \begin{pmatrix} 0&1 \\ 1 &0 \end{pmatrix} Y,
\end{gather}
defined over $K$. A fundamental matrix for the system (\ref{ex3}) is
\begin{gather*}
\begin{pmatrix} t_1 \sinh t_2 & t_1 \cosh t_2 \\ t_1 \cosh t_2 & t_1 \sinh t_2 \end{pmatrix}\!,
\end{gather*}
hence $L:=K\langle t_1 \sinh t_2,t_1 \cosh t_2 \rangle=K\langle \sinh t_2, \cosh t_2 \rangle=K\big\langle {\rm e}^{t_2} \big\rangle$ is a~Picard--Vessiot extension of $K$ for (\ref{ex3}).
An element $\sigma$ in the Galois group $\DGal(L/K)$ is determined by $\sigma\big({\rm e}^{t_2}\big)$ and $\sigma\big({\rm e}^{t_2}\big)=c {\rm e}^{t_2}$, with $c \in \R\setminus \{0\}$. We obtain then $\DGal(L/K) \simeq \G_m$.

We note that $\SO_2$ is isomorphic to $\G_m$ over $\C$ but not over $\R$. Considering the systems~(\ref{ex2}) and (\ref{ex3}) as defined over the field $\R(t_1,t_2)$ we obtain that the Picard--Vessiot extension for~(\ref{ex3}) is a real Liouvillian extension whereas the Picard--Vessiot extension for~(\ref{ex2}) is not a real gene\-ralised Liouvillian extension, since $\SO_2$ is not $\R$-split.\footnote{We note that Remark~3 in~\cite{GK} is incorrectly translated from the Russian original. The third sentence must read: ``This circle {\it does not have} a normal tower of subgroups with quotient groups isomorphic either to the additive or the multiplicative group of the field $\R$''.}
\end{Example}
\section{Comments on the integrability of real dynamical systems}

In this section we present the relationship between the concept of real Liouville function and the results on Liouvillian extensions of formally real differential fields treated in~\cite {CH} and Section~\ref{sLiou} of this paper. The solutions of differential equations which until now were abstract elements are considered in this section as real functions in real variables and we study their properties as such functions. Considering them as elements of a field of real functions, which is a formally real differential field, the results of the preceding sections apply. We comment as well on some questions related with the integrability of real dynamical systems going beyond the Liouviallinity questions. It is worth noting that real Liouville functions are a particular case of Pfaff functions which A. Khovanskii studies in his monograph~\cite{Kh}. However Pfaff functions may appear as solutions to non-linear differential equations and at present a Galois theory for non-linear differential equations defined over a formally real differential field has not been established. Khovanskii's theory uses essentially the completeness of the field $\R$ of real numbers. It~would be interesting to give an algebraic approach to this theory allowing to consider also real closed fields different from the field of real numbers.

{\sloppy
On the other hand, Theorem \ref{genL} shows that real Liouvillianity of an extension $L/K$ of formally real differential fields is codified in the structure of the differential Galois group $\DGal(L/K)$. A condition of this type is suggested by A.~Grothendieck in his proposal of an axiomatic definition of tame topology (cf.~\cite[p.~272]{GGA}). In our approach we use differential Galois groups defined over the real closed field of constants.

}

\subsection{Real Liouvillianity of solutions}
To any open and connected subset $U \subset \R^n$ we assign the integral domain of real analytic functions $\mathcal{O}(U)$ equipped with standard partial derivations $\partial_1 = \frac{\partial}{\partial x_1}, \dots ,\partial_n = \frac{\partial}{\partial x_n}$. Let $\mathcal{M}(U)$ denote the fraction field $\operatorname{Fr}(\mathcal{O}(U))$. We identify the elements of $\mathcal{M}(U)$ with the complete meromorphic fractions, i.e., the functions
\begin{gather*}
\phi \colon\ U \setminus \{g=0\}\ni x \mapsto \frac{f(x)}{g(x)} \in \R,
\end{gather*}
which are maximal elements with respect to inclusion.
From now on, we consider only differential fields which are subfields of differential fields of the form $(\mathcal{M}(U),\partial_1, \dots, \partial_n)$. Besides, we identify the field $\R$ with the trivial differential structure with the subfield of constant functions in $\mathcal{M}(U)$.

\begin{Definition} Let $f\colon U \to \R$ be a real analytic function defined in an open and connected subset $U \subset \R^n$. $f$ is called a \emph{real Liouville function} (resp.\ \emph{generalised real Liouville function}) if it lies in some formally real Liouvillian
extension (resp.\ generalised formally real Liouvillian extension) $K \subset \mathcal{M}(U)$ of the field $\R\subset \mathcal{M}(U)$.
\end{Definition}

\begin{Example} Let us consider the following dynamical system in $\mathbb{R}^2$.
\begin{gather*}
Y^{\prime}=\begin{pmatrix} 0 & 1 \\ 1/x^{2} & 0 \end{pmatrix}Y,
\end{gather*}
which corresponds to the second order linear differential equation $y^{\prime \prime
}-\big(1/x^{2}\big)y=0,$ defined over the differential field $K:=\mathbb{R}(x)$, with the usual derivation ${\rm d}/{\rm d}x$. Applying Kovacic's algorithm
(see~\cite[Section~7.3]{CHgsm}), we obtain that the two following real functions form a basis of the vector space of solutions
\begin{gather*}
y_{1}=x^{(1+\sqrt{5})/2}, \qquad y_{2}=x^{(1-\sqrt{5})/2}.
\end{gather*}
A Picard--Vessiot extension for this system over
$K$ is then $K\langle y_1, y_2 \rangle = K \langle y_1 \rangle$, since $y_{1}y_{2}=x\! \in\! K$. A $K$-differential isomorphism from $K \langle y_1 \rangle$ to $\C(x)\langle y_1 \rangle$ sends $y_{1}$ to $\lambda y_{1}$ with $\lambda \in \C^{\ast }$. Hence the differential Galois group is the multiplicative group $\mathbb{G}_m$, defined over $\R$.

Let us observe that the real Liouville functions $y_1$ and $y_2$ fulfill Pfaff equations of the form $Y'=\alpha \frac{Y}{x}$, where $\alpha= \frac 1 2 \big(1\pm \sqrt{5}\big)$ and therefore they can be studied by the methods of Pfaff geometry. More precisely the graphs of $y_{1}$ and $y_{2}$ are leaves of the foliations
in $\mathbb{R}^2$ defined by the 1-forms $\omega _{\alpha }=\alpha y{\rm d}x-x{\rm d}y$. Considering them in the positive half plane $\{x>0\}$ we obtain separating solutions to which Khovanskii's theory is aplicable (cf.\ \cite[pp.~4 and~5]{Kh}). In this paper we do not enter in details of this theory but an interested reader can consult \cite{Kh}.
\end{Example}

Our next example is related with real Liouville first integrals of simple gradient systems. Recently gradient systems are intensively studied in relation with the gradient conjecture and its generalisations (see~\cite{ARTH,CM}).

\begin{Example}\label{E}
T.H.~Colding and W.P.~Minicozzi II in their recent work~\cite{ARTH} have formulated a~generalised version of the famous Ren\'e Thom conjecture. More precisely Conjecture~1.1 from~\cite{ARTH} has been proved in~\cite{KK}, while the following conjecture \cite[Conjecture~1.2]{ARTH},
 called by the authors Arnold--Thom conjecture, remains open.

Let $f$ be a real analytic function in an open set $U \subset \R^n$ and $\grad f$ be its gradient in the Euclidean metric.

\begin{t*}
If a gradient flow line $x(t)$ has a limit point $x_0 \in U$, then the limit of the unit tangents $\frac{x'(t)}{|x(t)|}$ at $x_0$ exists.
\end{t*}

 In our second example, we present a polynomial gradient system that admits Liouville first integral. Its trajectories in the neighborhood of the origin are $C^1$ manifolds with boundary, but do not admit a $C^2$ extension through their limit point. Thus this example shows that the above statement of Arnold--Thom conjecture is optimal.

Let us consider a simple quadratic polynomial potential on $\mathbb{R}^2$
\begin{gather*}
f(x,y)=\lambda x^2+\mu y^2, \qquad \lambda,\mu \in \mathbb{R}\setminus \{0\}
\end{gather*}
and its gradient dynamical system
\begin{gather}
\frac{{\rm d}x}{{\rm d}t} = 2 \lambda x, \nonumber
\\
\frac{{\rm d}y}{{\rm d}t} = 2 \mu y.\label{grad}
\end{gather}
Solving (\ref{grad}), one obtains the following real Liouvillian first integral
\[
I(x,y)=\dfrac{x^{\mu}}{y^{\lambda}}.
\]
It is easy to see that for $\lambda=2$, $\mu=3$ and $I(x,y)=1$, the level curve is an ordinary cusp. Moreover as the curvature at the origin of $\big\{(x,y) \in \mathbb{R}^2 \colon y^2=x^3 \big\}$ tends to infinity, a $C^2$ prolongation of the branch curve $y=x^{3/2}$ is impossible (see, e.g.,~\cite[Theorem~5.1.6]{DFN}).
\end{Example}

\subsection{Final remarks}

The Arnold--Thom conjecture was known at the beginning of the 90's of last century under the name ``Tangent Problem'' (see~\cite[Section~9]{L}).
One more interesting conclusion suggested by the Example \ref{E} is that rational or, more generally, meromorphic integrability is quite exceptional even among systems that admit Liouville integrals. An interested reader can consult the excellent survey on the classical integrability problems~\cite{MP} and some more computational examples of non-integrable real dynamical systems in~\cite{HM}.

\section*{Acknowledgments}

We are very thankful to the anonymous referees for their valuable comments which helped us to improve significantly the presentation of our results.
R.~Mohseni acknowledges support of the Polish Ministry of Science and
Higher Education. T.~Crespo and Z.~Hajto acknowledge support of
grant PID2019-107297GB-I00 (MICINN).

\pdfbookmark[1]{References}{ref}
\LastPageEnding


\begin{thebibliography}{99}
\footnotesize\itemsep=0pt

\bibitem{BCR}
Bochnak J., Coste M., Roy M.-F., Real algebraic geometry, \textit{Ergebnisse der
 Mathematik und ihrer Grenzgebiete~(3)}, Vol.~36, \href{https://doi.org/10.1007/978-3-662-03718-8}{Springer-Verlag}, Berlin,
 1998.

\bibitem{B}
Borel A., Linear algebraic groups, 2nd ed., \textit{Graduate Texts in Mathematics}, Vol.~126, \href{https://doi.org/10.1007/978-1-4612-0941-6}{Springer-Verlag}, New York, 1991.

\bibitem{ARTH}
Colding T.H., Minicozzi II W.P., Arnold--{T}hom gradient conjecture for the
 arrival time, \href{https://doi.org/10.1002/cpa.21824}{\textit{Comm. Pure Appl. Math.}} \textbf{72} (2019), 1548--1577,
 \href{https://arxiv.org/abs/1712.05381}{arXiv:1712.05381}.

\bibitem{CM}
Colding T.H., Minicozzi~II W.P., Analytical properties for degenerate
 equations, in Geometric analysis, \textit{Progr. Math.}, Vol.~333,
 \href{https://doi.org/10.1007/978-3-030-34953-0_4}{Birkh\"auser/Springer}, Cham, 2020, 57--70, \href{https://arxiv.org/abs/1804.08999}{arXiv:1804.08999}.

\bibitem{CHgsm}
Crespo T., Hajto Z., Algebraic groups and differential {G}alois theory,
 \textit{Graduate Studies in Mathematics}, Vol.~122, \href{https://doi.org/10.1090/gsm/122}{Amer. Math. Soc.},
 Providence, RI, 2011.

\bibitem{CHJac}
Crespo T., Hajto Z., Picard--{V}essiot theory and the {J}acobian problem,
 \href{https://doi.org/10.1007/s11856-011-0145-y}{\textit{Israel~J. Math.}} \textbf{186} (2011), 401--406.

\bibitem{CH}
Crespo T., Hajto Z., Real {L}iouville extensions, \href{https://doi.org/10.1080/00927872.2014.888561}{\textit{Comm. Algebra}}
 \textbf{43} (2015), 2089--2093, \href{https://arxiv.org/abs/1206.2283}{arXiv:1206.2283}.

\bibitem{CHSA}
Crespo T., Hajto Z., Sowa-Adamus E., Galois correspondence theorem for
 {P}icard--{V}essiot extensions, \href{https://doi.org/10.1007/s40598-015-0029-z}{\textit{Arnold Math.~J.}} \textbf{2} (2016),
 21--27, \href{https://arxiv.org/abs/1502.08026}{arXiv:1502.08026}.

\bibitem{CHP}
Crespo T., Hajto Z., van~der Put M., Real and {$p$}-adic {P}icard--{V}essiot
 fields, \href{https://doi.org/10.1007/s00208-015-1272-2}{\textit{Math. Ann.}} \textbf{365} (2016), 93--103, \href{https://arxiv.org/abs/1307.2388}{arXiv:1307.2388}.

\bibitem{DFN}
Dubrovin B.A., Fomenko A.T., Novikov S.P., Modern geometry~-- methods and
 applications. {P}art~{I}. The geometry of surfaces, transformation groups, 2nd ed.,
 and fields, \textit{Graduate Texts in Mathematics}, Vol.~93,
 \href{https://doi.org/10.1007/978-1-4612-4398-4}{Springer-Verlag}, New York, 1992.

\bibitem{GK}
Gel'fond O.A., Khovanskii A.G., Real {L}iouville functions, \href{https://doi.org/10.1007/BF01086557}{\textit{Funct.
 Anal. Appl.}} \textbf{14} (1980), 122--123.

\bibitem{GGO}
Gillet H., Gorchinskiy S., Ovchinnikov A., Parameterized {P}icard--{V}essiot
 extensions and {A}tiyah extensions, \href{https://doi.org/10.1016/j.aim.2013.02.006}{\textit{Adv. Math.}} \textbf{238} (2013),
 322--411, \href{https://arxiv.org/abs/1110.3526}{arXiv:1110.3526}.

\bibitem{GGA}
Grothendieck A., Esquisse d'un programme, in Geometric {G}alois Actions,~1,
 Editors L.~Schneps, P.~Lochak, \textit{London Math. Soc. Lecture Note Ser.},
 Vol.~242, \href{https://doi.org/10.1017/CBO9780511758874.003}{Cambridge University Pres}s, Cambridge, 1997, 5--48, {E}nglish translation on pp.~243--283.

\bibitem{HM}
Hajto Z., Mohseni R., Tame topology and non-integrability of dynamical systems,
 \href{https://arxiv.org/abs/2008.12074}{arXiv:2008.12074}.

\bibitem{KP}
Kamensky M., Pillay A., Interpretations and differential {G}alois extensions,
 \href{https://doi.org/10.1093/imrn/rnw019}{\textit{Int. Math. Res. Not.}} \textbf{2016} (2016), 7390--7413.

\bibitem{Kh}
Khovanskii A.G., Fewnomials, \textit{Translations of Mathematical Monographs},
 Vol.~88, \href{https://doi.org/10.1090/mmono/088}{Amer. Math. Soc.}, Providence, RI, 1991.

\bibitem{Kol0}
Kolchin E.R., Algebraic matric groups and the {P}icard--{V}essiot theory of
 homogeneous linear ordinary differential equations, \href{https://doi.org/10.2307/1969111}{\textit{Ann. of Math.}}
 \textbf{49} (1948), 1--42.

\bibitem{Kol1}
Kolchin E.R., Picard--{V}essiot theory of partial differential fields,
 \href{https://doi.org/10.2307/2032594}{\textit{Proc. Amer. Math. Soc.}} \textbf{3} (1952), 596--603.

\bibitem{KK}
Kurdyka K., Mostowski T., Parusi\'nski A., Proof of the gradient conjecture of
 {R}.~{T}hom, \href{https://doi.org/10.2307/2661354}{\textit{Ann. of Math.}} \textbf{152} (2000), 763--792,
 \href{https://arxiv.org/abs/math.AG/9906212}{arXiv:math.AG/9906212}.

\bibitem{L}
{\L}ojasiewicz S., On semi-analytic and subanalytic geometry, in Panoramas of
 Mathematics ({W}arsaw, 1992/1994), \textit{Banach Center Publ.}, Vol.~34,
 Polish Acad. Sci. Inst. Math., Warsaw, 1995, 89--104.

\bibitem{MP}
Maciejewski A.J., Przybylska M., Differential {G}alois theory and
 integrability, \href{https://doi.org/10.1142/S0219887809004272}{\textit{Int.~J. Geom. Methods Mod. Phys.}} \textbf{6} (2009),
 1357--1390, \href{https://arxiv.org/abs/0912.1046}{arXiv:0912.1046}.

\bibitem{Pre}
Prestel A., Lectures on formally real fields, \textit{Lecture Notes in Math.},
 Vol.~1093, \href{https://doi.org/10.1007/BFb0101548}{Springer-Verlag}, Berlin, 1984.

\bibitem{PR}
Prestel A., Roquette P., Formally {$p$}-adic fields, \textit{Lecture Notes in
 Math.}, Vol.~1050, \href{https://doi.org/10.1007/BFb0071461}{Springer-Verlag}, Berlin, 1984.

\bibitem{PS}
van~der Put M., Singer M.F., Galois theory of linear differential equations,
 \textit{Grundlehren der mathematischen Wissenschaften}, Vol.~328,
 \href{https://doi.org/10.1007/978-3-642-55750-7}{Springer-Verlag}, Berlin, 2003.

\end{thebibliography}
\end{document}